\documentclass[12pt]{amsart}
\usepackage[dvipsnames,usenames]{color}

\usepackage{amsmath,amsthm,amssymb}
\usepackage{times}
\usepackage{enumerate}
\usepackage{amsfonts, hyperref, graphicx, ifpdf}

\usepackage[dvipsnames,usenames]{color}

\newtheorem{theorem}{Theorem}[section]

\newtheorem{lemma}[theorem]{Lemma}

\newcommand{\be}{\begin{equation}}
\newcommand{\ee}{\end{equation}}
\newcommand{\lt}{\left}
\newcommand{\rt}{\right}
\newcommand{\R}{\mathbb{R}}
\newcommand{\s}{\mathbb{S}}
\newcommand{\al}{\alpha}

\newcommand{\mL}{\mathcal{L}}

\newcommand{\bn}{\bar{\nabla}}
\newcommand{\goto}{\rightarrow}
\newcommand{\mM}{\mathcal{M}}

\theoremstyle{definition}
\newtheorem{defin}[theorem]{Definition}

\numberwithin{equation}{section}

\begin{document}
\setlength{\baselineskip}{1.2\baselineskip}

\title[Asymptotic convergence for modified scalar curvature flow]
{Asymptotic convergence for modified scalar curvature flow}

\author{Ling Xiao}
\address{Department of Mathematics, University of Connecticut,
Storrs, CT 06268}
\email{ling.2.xiao@uconn.edu}

\begin{abstract}
In this paper, we study the flow of closed, starshaped hypersurfaces in $\mathbb{R}^{n+1}$
with speed $r^\al\sigma_2^{1/2},$ where $\sigma_2^{1/2}$ is the normalized square root of the
scalar curvature, $\al\geq 2,$ and $r$ is the distance from points on the hypersurface to the origin.
We prove that the flow exists for all time and the starshapedness is preserved.
Moreover, after normalization, we show that the flow converges exponentially fast to a sphere centered at origin.
When $\al<2,$ a counterexample is given for the above convergence.
\end{abstract}

\maketitle

\section{Introduction}
\label{sec0}

In this paper, we will consider the evolution of a compact, starshaped hypersurface $\Sigma_0\subset\R^{n+1}$ by modified
 scalar curvature. Namely, we will study the following geometric flow:
\be\label{0.1}
\lt\{
\begin{aligned}
&\frac{\partial X}{\partial t}(x, t)=-r^{\al}\sigma_2^{1/2}(x, t)\nu\\
&X(x, 0)=X_0(x),\\
\end{aligned}
\right.
\ee
where $\al\geq 2, $ $\sigma_2=\binom{n}{2}^{-1}S_2=\binom{n}{2}^{-1}\sum\limits_{1\leq i_1<i_2\leq n}\kappa_{i_1}\kappa_{i_2}$
is the scalar curvature of the hypersurface $\Sigma_t, $ parametrized by $X(\cdot, t):\s^n\goto\R^{n+1},$ $r=|X(x, t)|,$ and $\nu(\cdot, t)$ is the unit outer normal
at $X(\cdot, t).$ Following \cite{CNS3}, we make the following definition:
\begin{defin}
\label{def0.1}
A hypersurface $\Sigma$ is called {\em 2-convex}, if for any $p\in\Sigma,$ the principal curvatures of $\Sigma$ at $p$ satisfy
\[\kappa[\Sigma(p)]=(\kappa_1,\cdots, \kappa_n)\in\Gamma_2,\]
where $\Gamma_2$ is the G\r{a}rding cone:
\[\Gamma_2=\{\lambda\in\R^n|\sigma_1(\lambda)>0\,\, \mbox{and $\sigma_2(\lambda)>0$}\}.\]
\end{defin}

Flows of hypersurfaces in Euclidean space $\R^{n+1}$ by functions of principal curvatures have been extensively studied in the past
four decades. The flow generated by Gauss curvature was first introduced by Firey \cite{Fir74} as a model for the changing shape of a tumbling stone,
which is subject to collisions from all directions with uniform frequency.
Since then, the Gauss curvature flow has been studied by many authors (see \cite{And99}, \cite{And00}, \cite{BCD17},
\cite{Cho85}, \cite{GN17}, \cite{DL04}, etc.). In particular, Andrews \cite{And99} proved that the Gauss curvature flow deforms a uniformly convex hypersuface into a round point when $n=2.$ In higher dimensions, the corresponding result is obtained by combining the results in \cite{GN17} and \cite{BCD17}. Such properties are the generalizations of Huisken \cite{Hui84} for the mean curvature flow.

As a natural extension, the study of different types of fully nonlinear geometric flows, especially their asymptotic behaviors, have attracted lots of attentions through the years (see \cite{And94}, \cite{BH17}, \cite{CZ99}, \cite{Gag93}, \cite{Ger90}, \cite{Sch06}, \cite{Urb90}, \cite{Urb91}, etc.). In a recent paper \cite{LSW17},  Li, Sheng, and Wang studied a contracting flow with speed $fr^\al K,$ where $K$ is the Gauss curvature and $f$ is a positive function. They provided a parabolic proof for the classical Aleksandrov problem; they also resolved the dual $q$- Minkowski problem for the case $q<0.$ Moreover,
in their follow-up paper \cite{LSW18}, they studied the evolution of closed, convex hypersufaces with speed
$r^\alpha\sigma_k,$ where $\sigma_k$ is the $k$-th elementary symmetric polynomial of principal curvatures.

The flow problem \eqref{0.1} we study here is inspired by \cite{LSW17, LSW18}. Generally speaking, it is more difficult to study fully nonlinear curvature flows and their asymptotic behaviors for the following reasons. First, there is a lack of monotonicity quantities; usually in order to study the asymptotic behavior of curvature flows it is nice to have some monotonicity formulas (see \cite{Hui90} and \cite{GN17} for example). Second, there is a lack of convexity estimates in the limit; unlike in \cite{HS99, BH17, ALM15}, we do not know if the limit is convex. Lastly, the admissible cone for general curvature flows can be very large, which makes it hard to obtain curvature estimates.

Our motivations for considering geometric flows of this type are the following: (i) It is the first step in solving the Christoffel-Minkowski problem for
curvature measures using a flow approach. (ii) We introduce new techniques to obtain the curvature estimates, and we expect these new techniques will be useful for other geometry problems.

Let us state our main result as follows.
\begin{theorem}
\label{th0.1} Let $\mathcal{M}_0$ be a smooth, closed, 2-convex, starshaped hypersurface in $\R^{n+1}.$ Then the
flow \eqref{0.1} has a unique smooth starshaped solution $\mathcal{M}_t$ with positive scalar curvature, for all time $t>0,$ which converges to the origin. After a proper rescaling,
$X\goto\phi^{-1}(t)X,$ the hypersurface $\tilde{\mM}_t=\phi^{-1}(t)\mM_t$ converges exponentially fast to a sphere centered at the origin in the $C^\infty$ topology.
\end{theorem}
Following \cite{LSW17}, our choice of the rescaling factor $\phi(t)$ is motivated by the calculations below.
Assume
\be\label{0.2}
X(\cdot, t)=\phi(t)X_0(\cdot)
\ee
evolves under the flow \eqref{0.1} with initial data $\phi_0X_0,$ where $\phi$ is a positive function and $\phi_0=\phi(0).$ By differentiating equation
\eqref{0.2} with respect to $t$ and multiplying both sides by $\nu_0=\nu(\cdot, t),$ we get
\be\label{0.3}
\phi'(t)\lt< X_0, \nu_0\rt>=-\phi^{\al-1}\sigma_2^{1/2}(0)r_0^\al,
\ee
where $\sigma_2^{1/2}(0)$ is the normalized scalar curvature of $M_0=X(\mathbb{S}^n),$ and $r_0$ is the radial function of $M_0.$
By \eqref{0.3} we have
\[\phi'=-\lambda\phi^{\al-1}\]
for some constant $\lambda>0.$ We may assume $\lambda=1.$ Then
\be\label{0.4}
\begin{aligned}
\phi(t)&=\phi_0e^{-t},\,\,\mbox{if $\al=2,$}\\
\phi(t)&=\lt(\phi_0^{2-\al}-(2-\al)t\rt)^{\frac{1}{2-\al}},\,\,\mbox{if $\al>2$}.
\end{aligned}
\ee

The study of the asymptotic behavior of the flow \eqref{0.1} is equivalent to the study of the long time behavior of the rescaled flow \eqref{0.5}.
Let
\[\tilde{X}(\cdot, \tau)=\phi^{-1}(t)X(\cdot, t),\]
where $\tau=-\ln\phi.$ Then $\tilde{X}(\cdot, \tau)$ satisfies the following equation
\be\label{0.5}
\lt\{
\begin{aligned}
\frac{\partial X}{\partial t}(x, t)&=-\sigma_2^{1/2}r^\al\nu+X,\\
X(\cdot, 0)&=X_0.
\end{aligned}
\rt.
\ee
For convenience, we still use $t$ instead of $\tau$ to denote the time variable, and without causing confusions, we omit ``tilde".

Theorem \ref{th0.1} is optimal in the following sense. If $\al<2,$ we find that the hypersurface evolving by \eqref{0.1} can reach the origin in finite time; hence, the flow does not converge to a round sphere centered at the origin. More precisely, we prove
\begin{theorem}
\label{th0.2}
Suppose $\al<2,$ then there exists a smooth, closed, uniformly convex (automatically starshaped and 2-convex) $\mM_0,$
such that under the flow \eqref{0.1},
\[\mathcal{R}(X(\cdot, t)):=\frac{\max_{\s^n}r(\cdot, t)}{\min_{\s^n}r(\cdot, t)}\goto\infty,\,\,\mbox{as $t\goto T$},\]
for some finite $T>0.$
\end{theorem}

The organization of the paper is as follows. In Section \ref{secp} we introduce some basic notations
and establish evolution equations for basic geometric quantities. In Section \ref{sec2.0} we derive $C^0$ and $C^1$ estimates. We show that $\sigma_2^{1/2}$ is bounded along the normalized flow \eqref{0.5} in Section \ref{sec3.0}. Section \ref{sec4} is devoted to the proof of our main a priori estimate, Theorem \ref{th4.1}. This estimate
shows that the principal curvatures stay bounded under flow \eqref{0.5}. We obtain the convergence result in Section \ref{sec5}, which completes the proof
of Theorem \ref{th0.1}. In Section \ref{sec6} we give a counter example by showing that if $\alpha<2,$ the flow \eqref{0.1} may have unbounded ratio of radii.
This proves Theorem \ref{th0.2}.

\section*{Acknowledgements}
The author would like to thank the referee for his/her careful reading and helpful comments. In particular, the author would like to thank the referee for pointing
out a mistake in the proof of Theorem \ref{th4.1}.

\bigskip

\section{Preliminaries}
\label{secp}

Let us first recall some basic properties of a starshaped hypersurface $\mathcal{M}$ in $\mathbb{R}^{n+1}.$ Since $\mathcal {M}$ is starshaped,
for a suitable diffeomorphism $\xi(\cdot): \mathbb{S}^n\goto\mathbb{S}^n,$ the position vector of $\mathcal{M}$ can be written as
\[X(x)= r(\xi(x))\xi(x),\]
where $r(\xi(x))=|X(x)|$ is the radial function. Next, we will give the expressions of the induced metric, second fundamental form, and Weingarten curvatures of
$\mathcal{M}$ in terms of the radial function.

Let $e_1,\cdots, e_n$ be a smooth local orthonormal frame on $\mathbb{S}^n,$ and let
$\bn$ be the covariant derivative on $\s^n.$ Then in term of $r,$ we have
\[g_{ij}=r^2\delta_{ij}+\bn_{i}r\bn_{j}r,\]
\[g^{ij}=r^{-2}\lt(\delta_{ij}-\frac{\bn_{i}r\bn_{j}r}{r^2+|\bn r|^2}\rt),\]
\[\nu=\frac{r\xi-\bn r}{\sqrt{r^2+|\bn r|^2}},\]
\[h_{ij}=\frac{1}{\sqrt{r^2+|\bn r|^2}}\lt(r^2\delta_{ij}+2\bn_ir\bn_jr-r\bn_{ij}r\rt).\]
The principal curvatures of $\mathcal{M}$ are the eigenvalues of $h_{ij}$ with respect to $g_{ij},$ namely the solutions of
\[0=\det(h_{ij}-\kappa g_{ij})=\det(a_{ij}-\kappa\delta_{ij}),\]
where   $a_{ij}=\lt(g^{-\frac{1}{2}}\rt)^{il}h_{lm}\lt(g^{-\frac{1}{2}}\rt)^{mj},$ and
\[\lt(g^{-\frac{1}{2}}\rt)^{ij}=r^{-1}\lt[\delta_{ij}-\frac{\bn_ir\bn_jr}{\sqrt{r^2+|\bn r|^2}(r+\sqrt{r^2+|\bn r|^2})}\rt]\] is the square root
of $g^{ij}.$ It is easy to see that if $\mathcal{M}_t$ is a family of starshaped hypersurfaces satisfying \eqref{0.5}, then
the radial function $r$ satisfies
\be\label{5.1}
\lt\{
\begin{aligned}
&r_t=-r^{\al-1}w\sigma_2^{1/2}+r,\\
&r(\cdot, 0)=r_0,\\
\end{aligned}
\rt.
\ee
where $r_0$ is the radial function of $\mathcal{M}_0,$ and $w=\sqrt{r^2+|\bn r|^2}.$

Next, we will derive some evolution equations for our normalized flow \eqref{0.5}.
We will use $\nabla$ to denote the covariant derivative with respect to the metric on $\mathcal{M}_t.$ Let $\{\tau_1, \cdots, \tau_n\}$
be a local orthonormal frame on $T\mathcal{M}_t,$ and recall the following identities:
\be\label{2.1.1}
X_{ij}=-h_{ij}\nu\,\,\,\, \mbox{(Gauss formula),}
\ee
\be\label{2.1.2}
\nu_i=h^l_i\tau_l\,\,\,\,\mbox{(Weingarten equation),}
\ee
and
\be\label{2.1.3}
\nabla_{sr}h_{ij}=\nabla_{ij}h_{rs}+h_{il}h^l_jh_{rs}-h_{ir}h_{sm}h^m_j+h_{sj}h^n_ih_{nr}-h_{ij}h^m_rh_{ms}\,\,\mbox{(Ricci identity)}.
\ee

We now consider equation \eqref{0.5} on $\mM_t.$ Let $\mathcal{A}$ be the vector space of $n\times n$ matrices and
\[\mathcal{A}_2=\{A=\{a_{ij}\}\in\mathcal{A}: \lambda(A)\in\Gamma_2\},\]
where $\lambda(A)=(\lambda_1, \cdots, \lambda_n)$ denotes the eigenvalues of $A.$ Let $F$ be the function defined by
$F(A)=f(\lambda(A))=\sigma_2^{1/2}(\lambda(A)),$ $A\in\mathcal{A}_2.$ In the rest of the paper, we will use the following notations,
\[F^{ij}=\frac{\partial F}{\partial a_{ij}}(A),\,\, F^{ij, kl}=\frac{\partial^2F}{\partial a_{ij}\partial a_{kl}}(A).\]
Since $F(A)$ depends only on the eigenvalues of $A,$ if $A$ is symmetric, so is $\{F^{ij}(A)\}.$ Moreover, when $A$ is diagonal,
$F^{ij}(A)=\frac{\partial f}{\partial\lambda_i}\delta_{ij}=f^i\delta_{ij}.$

\begin{lemma}
\label{lem2.1}
Denote $\Phi=r^\al\sigma_2^{1/2}$ and $F=\sigma_2^{1/2},$ then under the normalized flow \eqref{0.5}, we
have
\be\label{2.1.5}
\frac{\partial}{\partial t}g_{ij}=-2\Phi h_{ij}+2g_{ij},
\ee
\be\label{2.1.6}
\frac{\partial}{\partial t}\nu=g^{kl}\Phi_k\tau_l,
\ee
and
\be\label{2.1.7}
\frac{\partial}{\partial t}h^j_i=\Phi_i^j+\Phi h^k_ih^j_k-h^j_i.
\ee
\end{lemma}
\begin{proof}
By the Weingarten equation, we have
\begin{align*}
\frac{\partial}{\partial t}g_{ij}&=\lt<(X_t)_i, X_j\rt>+\lt<X_i, (X_t)_j\rt>\\
&=2\lt<-\Phi_i\nu-\Phi h^k_i\tau_k+\tau_i, \tau_j\rt>\\
&=-2\Phi h_{ij}+2g_{ij}.
\end{align*}
This proves \eqref{2.1.5}.

To derive \eqref{2.1.6} one observes that
\[\partial_t\lt<\nu, \tau_i\rt>=\lt<\nu_t, \tau_i\rt>+\lt<\nu, (X_i)_t\rt>=0.\]
Therefore,
\[\lt<\nu_t, \tau_i\rt>=-\lt<\nu, (-\Phi)_i\nu-\Phi h^k_i\tau_k+\tau_i\rt>=\Phi_i,\]
which implies that
\[\nu_t=g^{kl}\Phi_k\tau_l.\]
Next, we differentiate equation \eqref{2.1.6} with respect to $\tau_i$ and get
\[\partial_t\nu_i=\lt(g^{kl}\Phi_k\tau_l\rt)_i=\partial_t\lt(h^k_i\tau_k\rt).\]
Thus,
\begin{align*}
&\lt(\partial_th^k_i\rt)\tau_k+h^k_i\lt(-\Phi\nu+X\rt)_k\\
&=g^{kl}\Phi_{ki}\tau_l-g^{kl}\Phi_kh_{il}\nu.\\
\end{align*}
This gives us that
\[\partial_th^k_ig_{kj}+h^k_i\lt<-\Phi_k\nu-\Phi h^r_k\tau_r+\tau_k, \tau_j\rt>=g^{kl}\Phi_{ki}g_{lj},\]
which implies,
\[\partial_th^j_i=\Phi^j_i+\Phi h^k_ih^j_k-h^j_i.\]
\end{proof}

\bigskip

\section{$C^0$ and $C^1$ estimates}
\label{sec2.0}
In this section, we will establish the $C^0$ and $C^1$ estimates. In particular, we will show that the flow \eqref{0.5}
preserves the starshapedness of the initial hypersurface $\mM_0$. We also want to point out that throughout this paper, if not further specified,
 we will denote by $C$ and $C_i$ for $i\in\mathbb{N},$ some positive constant, whose value may change from line to line.
\subsection{$C^0$ estimates}
In this subsection, we will establish the uniform upper and lower bounds for the radial function $r$ of the normalized flow \eqref{5.1}.
\begin{lemma}
\label{lem1.1}
Let $r(\cdot, t)$ be a positive, 2-convex smooth solution to \eqref{5.1} on $\mathbb{S}^n\times[0, T).$
If $\al\geq2,$ then there exists a positive constant  $C$ depending only on $\max_{\mathbb{S}^n}r(\cdot, 0)$
and $\min_{\mathbb{S}^n}r(\cdot, 0)$ such that
\be\label{1.2}
1/C\leq r(\cdot, t)\leq C,\,\,\forall t\in[0, T).
\ee
\end{lemma}
\begin{proof}
Let $r_{\min}(t)=\min\limits_{\s^n}r(\cdot, t),\,\,\forall t\in(0, T).$ Note that at the point where $r_{\min}(t)$ is achieved, we have
 $\sigma_2^{1/2}\leq\frac{1}{r_{\min}(t)}.$ Then by equation \eqref{5.1} we get
\[\frac{d}{dt}r_{\min}\geq r_{\min}\lt(1-r_{\min}^{\al-2}\rt).\]
Therefore, when $\alpha=2,$ we get
\[\frac{d}{dt}r_{\min}\geq 0,\]
which yields $r(\cdot, t)\geq \min r(\cdot, 0), \forall t\in (0, T);$
when $\alpha>2,$ we may assume $\min r(\cdot, t)<1,$ otherwise we would be done. It follows that
\[\frac{d}{dt}r_{\min}\geq 0.\] This implies when $\alpha\geq 2$ we have
\[r(\cdot, t)\geq\min\{1, \min\limits_{\s^n}r(\cdot, 0)\}.\]
Similarly, we have
\[r(\cdot, t)\leq\max\{1, \max\limits_{\s^n}r(\cdot, 0)\}.\]
\end{proof}

\subsection{$C^1$ estimates}
In this subsection, we will apply the evolution equations derived in Section \ref{secp} to obtain the gradient estimate. This result yields that, if we start from a starshaped hypersurface $\mM_0,$ then as long as the flow exists, $\mM_t$ remains starshaped.
\label{sub2.2.0}
\begin{lemma}
\label{lem2.2}
Let $X(\cdot, t)$ be a family of smooth, 2-convex hypersurfaces that solves the normalized flow \eqref{0.5} on $\mathbb{S}^{n}\times [0, T).$
Denote $u=\lt<X, \nu\rt>,$ then there exists a constant $C$ depending on $\mM_0,$ $|X|_{C^0},$ and $n$ such that
\be\label{2.2.1}
1/C< u<C,\,\,t\in[0, T).
\ee
\end{lemma}
\begin{proof}
The upper bound of $u$ is a direct consequence of Lemma \ref{lem1.1}. Therefore, in the following, we only need to show $u$ is bounded from below.
We will follow \cite{GLM09} considering
\[P=\gamma(S)-\log\lt<X, \nu\rt>,\]
where $S=\lt<X, X\rt>,$ $\gamma(S)=\frac{\lambda}{S},$ and $\lambda>0$ to be determined.

Assume $P$ achieves its maximum at an interior point $X_0\in \mathcal{M}_{t_0}.$ In the following, all calculations will be done at this point with respect to a local orthonormal frame $\tau_1, \cdots, \tau_n$. We can see that at $X_0$
\[P_i=\gamma'S_i-\frac{u_i}{u}=0.\]
By a straightforward calculation we get,
\[S_i=2\lt<X, \tau_i\rt>,\]
\[S_{ij}=2\lt<\tau_i, \tau_j\rt>-2\lt<X, h_{ij}\nu\rt>=2\delta_{ij}-2h_{ij}u,\]
and
\[S_t=2\lt<X, X_t\rt>=2\lt<X, -\Phi\nu+X\rt>=-2\Phi u+2S.\]
Therefore,
\be\label{2.2.3}
\begin{aligned}
\mL S&=S_t-r^\al F^{ij}S_{ij}\\
&=-2\Phi u+2S-r^\al F^{ij}\lt(2\delta_{ij}-2h_{ij}u\rt)\\
&=2S-2r^\al\sum f^i.
\end{aligned}
\ee
Moreover,
\begin{align*}
u_t&=\lt<X_t, \nu\rt>+\lt<X, \nu_t\rt>\\
&=\lt<-\Phi\nu+X, \nu\rt>+\lt<X, \nabla\Phi\rt>\\
&=-\Phi+u+\lt<X, \al r^{\al-1}F\nabla r+r^\al\nabla F\rt>,\\
\end{align*}
\[u_i=\lt<X_i, \nu\rt>+\lt<X, h_{ik}\tau_k\rt>=h_{ik}\lt<X, \tau_k\rt>,\]
and
\[u_{ij}=h_{ij}+\lt<X, h_{ijk}\tau_k\rt>-h_{ik}h_{kj}u.\]
Therefore,
\be\label{2.2.4}
\begin{aligned}
\mL u&=u_t-r^{\al}F^{ij}u_{ij}\\
&=-\Phi+u+\al r^{\al-1}f\lt<X, \nabla r\rt>+r^\al\lt<X, \nabla f\rt>-r^\al f-r^\al\lt<X, \nabla f\rt>
+r^\al u\sum  F^{ij}h_{ik}h_{kj}\\
&=-2\Phi+u+\al r^{\al-1}f\lt<X, \nabla r\rt>+r^\al u\sum F^{ij}h_{ik}h_{kj}.
\end{aligned}
\ee

Since at $X_0$ we have $P_i=0,$ it follows that
\be\label{3.3*}
2\gamma'\lt<X, \tau_i\rt>=\frac{h_{ik}\lt<X, \tau_k\rt>}{u}.
\ee
If at this point, $\lt<X, \tau_i\rt>=0$ for all $1\leq i\leq n,$ we would get
$\lt<X,\nu\rt>^2=|X|^2,$ then by Lemma \ref{lem1.1} we are done. So we may assume
$\lt<X,\nu\rt>^2<|X|^2$ at $X_0.$ We may also choose a smooth local orthonormal frame on $\mathcal{M}_{t_0}$
such that at $X_0,$ $\lt<X, \tau_i\rt>=0,\,\,i\geq 2.$
Therefore at this point, we have $h_{11}=2\gamma'u$ and $h_{1i}=0$ for $i\geq 2.$
We may also rotate
$\{\tau_2, \cdots, \tau_n\}$ such that $h_{ij}=\kappa_i\delta_{ij}$ is diagonal.

Next, we compute
\[P_t=\gamma' S_t-\frac{u_t}{u},\]
and
\[P_{ij}=\gamma'S_{ij}+\gamma''S_iS_j-\frac{u_{ij}}{u}+\frac{u_iu_j}{u^2}.\]
Therefore
\be\label{2.2.2}
\begin{aligned}
\mL P&=P_t-r^\al F^{ij}P_{ij}\\
&=\gamma'\mL S-\frac{1}{u}\mL u-r^\al\gamma'' F^{ij}S_iS_j-r^\al F^{ij}\frac{u_iu_j}{u^2}.
\end{aligned}
\ee
Substituting equations \eqref{2.2.3} and \eqref{2.2.4} into \eqref{2.2.2} and
applying the maximum principle we obtain at $X_0$,
\begin{align*}
&\gamma'\lt(2S-2r^\al\sum f^i\rt)-\frac{1}{u}\lt(-2\Phi+u+\al r^{\al-1}f\lt<X,\nabla r\rt>+r^\al u\sum f^i\kappa_i^2\rt)\\
&-r^\al\gamma''f^iS_i^2-r^\al f^i\frac{u_i^2}{u^2}\geq 0.\\
\end{align*}
Since $r^2=S,$ we get $2rr_i=2\lt<X, \tau_i\rt>,$ and
\[\lt<X, \nabla r\rt>=\frac{\lt<X, \tau_1\rt>^2}{r}=\frac{S-u^2}{r}=r-\frac{u^2}{r}.\]
Hence, we get
\be\label{2.2.5}
\begin{aligned}
&2\gamma'S+2\frac{\Phi}{u}-\frac{\al r^{\al-1}f}{u}\lt(r-\frac{u^2}{r}\rt)\\
&\geq 2r^\al\gamma'\sum f^i+r^\al\sum f^i\kappa_i^2+4r^\al\lt[\lt(\gamma''+(\gamma')^2\rt)f^1\lt(S-u^2\rt)\rt]+1,\\
\end{aligned}
\ee
where we used equation $\gamma'S_i=\frac{u_i}{u}$.
Substituting $\gamma=\frac{\lambda}{S}$ into \eqref{2.2.5} we obtain
\begin{align*}
&-\frac{2\lambda}{S}+\al r^{\al-2}fu+\frac{(2-\al)\Phi}{u}+4r^\al\lt(\frac{2\lambda}{S^3}+\frac{\lambda^2}{S^4}\rt)f^1u^2\\
&\geq 1-2\frac{r^\al\lambda}{S^2}\sum f^i+r^\al\sum f^i\kappa_i^2+4r^\al\lt(\frac{2\lambda}{S^3}+\frac{\lambda^2}{S^4}\rt)f^1S.\\
\end{align*}

First, since $\al\geq 2,$ we have $\frac{(2-\al)\Phi}{u}\leq 0.$
Moreover, differentiating
\[\beta f^2=\beta\sigma_2=S_2,\,\,\mbox{where $\beta= {n\choose 2},$}\] with respect to $\lambda_i$ we get
\[2\beta ff^i=S_1(\kappa|i)=\sum\limits_{j\neq i}\kappa_j.\]
This together with $h_{11}=-2\frac{\lambda}{S^2}u<0$ implies
\[2\beta f\sum f^i=(n-1)S_1(\kappa)<(n-1)S_1(\kappa|1)=2(n-1)\beta f f^1.\]
Therefore, we get
\[1\leq\sum f^i<(n-1)f^1,\]
where the first inequality comes from the concavity of $f$ (see \cite{GS09}).
Moreover, we can always choose $\lambda>0$ large such that
\[\frac{-2r^\al\lambda}{S^2}(n-1)f^1+2r^\al\lt(\frac{2\lambda}{S^3}+\frac{\lambda^2}{S^4}\rt)f^1S>0.\]
Thus we have
\be\label{2.2.6}
\begin{aligned}
&\al r^{\al-2}fu+4r^\al\lt(\frac{2\lambda}{S^3}+\frac{\lambda^2}{S^4}\rt)f^1u^2\\
&\geq 1+\frac{2\lambda}{S}+r^\al\sum f^i\kappa_i^2+2r^\al\lt(\frac{2\lambda}{S^3}+\frac{\lambda^2}{S^4}\rt)f^1S.
\end{aligned}
\ee
Since $f$ is concave, we have
\be\label{2.2.6'}
\sum f^i\kappa_i^2\geq f(\kappa_1^2, \cdots, \kappa_n^2)>c(n)f^2.
\ee
Case 1. When
$\al r^{\al-2}fu\geq 1+ \frac{2\lambda}{S}+r^\al\sum f^i\kappa_i^2\geq 1+\frac{2\lambda}{S}+c(n)f^2,$ we get
\[u\geq\frac{C_1}{f}+c(n)f>C_2.\]
Case 2. When $\al r^{\al-2}fu< 1+\frac{2\lambda}{S}+r^\al\sum f^i\kappa_i^2,$ we get
\[4r^\al\lt(\frac{2\lambda}{S^3}+\frac{\lambda^2}{S^4}\rt)f^1u^2>2r^\al\lt(\frac{2\lambda}{S^3}+\frac{\lambda^2}{S^4}\rt)f^1S.\]
This gives $u^2>\frac{S}{2}.$

Combining case 1 and case 2 we  conclude that $u$ is bounded from below at $X_0$, which in turn implies that $u$ is bounded from below everywhere.
Hence, we proved this lemma.
\end{proof}
For later usage, we want to point out that Lemma \ref{lem2.2} implies that $|\nabla r(\cdot, t)|<C,$ for $t\in [0, T).$

\bigskip
\section{Bound on $F$}
\label{sec3.0}

In this section we will show that along the flow, $F=\sigma_2^{1/2}$ is bounded from above and below.
\begin{lemma}
\label{lem3.1}
Under the normalized flow \eqref{0.5}, there exists a constant $C$ depending only on $\mM_0$ and $r$,such that
\be\label{3.1}
F>\frac{1}{C}.
\ee
\end{lemma}
\begin{proof}
Let $\Phi=r^\al F,$
recall the evolution equation \eqref{2.1.7} we have
\be\label{3.3}
F_t=F^{ij}\lt(\Phi_{ij}+\Phi h_i^kh^j_k-h^j_i\rt).
\ee
Moreover, by \eqref{0.5} we get
\[r_t=\frac{\lt<X_t, X\rt>}{r}=-\frac{\Phi u}{r}+r.\]
Therefore, choosing an orthonormal frame such that $h^j_i=\kappa_i\delta_{ij}$ we get
\be\label{3.4}
\begin{aligned}
\frac{\partial\Phi}{\partial t}&=\al r^{\al-1}r_tF+r^\al F_t\\
&=\frac{\al\Phi}{r}\lt(-\frac{\Phi u}{r}+r\rt)+r^\al F^{ij}\lt(\Phi_{ij}+\Phi h^k_ih^j_k-h_i^j\rt)\\
&=\frac{\alpha\Phi}{r}\lt(-\frac{\Phi u}{r}+r\rt)+r^\al F^{ii}\Phi_{ii}+r^\alpha\Phi\sum f^i\kappa_i^2-\Phi.
\end{aligned}
\ee

Let $\Phi_{\min}(t)=\min\limits_{x\in\s^n}\Phi(x, t),$ then $\Phi_{\min}$ satisfies
\[\frac{d}{dt}\Phi_{\min}\geq-\frac{\al\Phi_{\min}^2u}{r}+(\al-1)\Phi_{\min}+r^\al\Phi_{\min}\sum f^i\kappa_i^2.\]
Thus we have
\[\frac{d}{dt}\Phi_{\min}\geq\Phi_{\min}\lt[(\al-1)-\frac{\al u}{r}\Phi_{\min}\rt].\]
We can see that, when $\Phi_{\min}<\frac{r(\al-1)}{u\al},$ then $\frac{d}{dt}\Phi_{\min}\geq 0.$
Therefore, we conclude that
\[\Phi_{\min}\geq\min\lt\{\min\limits_{\s^n}\Phi(\cdot, 0), \min\limits_{\s^n\times[0, T)}\frac{r(\al-1)}{u\al}\rt\}.\]
Together with Lemma \ref{lem1.1} and Lemma \ref{lem2.2} we get $F$ is bounded from below.
\end{proof}

\begin{lemma}
\label{lem3.2}
Under the normalized flow \eqref{0.5}, there exists a constant $C$ depending on $\mM_0$, $u,$ and $r$ such that
\be\label{3.2}
F<C.
\ee
\end{lemma}
\begin{proof}
Let $\mL:=\frac{\partial}{\partial t}-r^{\al} F^{ij}\nabla_{ij},$ by \eqref{3.4} we get
\be\label{3.5}
\mL\Phi=-\frac{\alpha\Phi^2u}{r^2}+\al\Phi+\Phi r^\al\sum f^i\kappa_i^2-\Phi.
\ee
We also recall that
\be\label{3.6}
\mL u=-2\Phi+u+\al r^{\al-1}F\lt<X, \nabla r\rt>+r^\al u\sum f^i\kappa_i^2.
\ee
Considering $M=\log\Phi-\log(u-a),$ where $a=\frac{1}{2}\min\limits_{\s^n\times[0, T)} u.$  At its maximum point, by equations \eqref{3.5} and \eqref{3.6} we have
\be\label{3.7}
\begin{aligned}
\mL M&=\frac{\mL \Phi}{\Phi}-\frac{\mL u}{u-a}\\
&=-\frac{\al\Phi u}{r^2}+(\al-1)+r^\al\sum f^i\kappa_i^2+\frac{2\Phi}{u-a}\\
&-\frac{u}{u-a}-\frac{\al\Phi}{r(u-a)}\lt<X,\nabla r\rt>
-r^\al\frac{u}{u-a}\sum f^i\kappa_i^2\geq 0.
\end{aligned}
\ee
Applying Lemmas \ref{lem1.1} and \ref{lem2.2} we get
\be\label{3.8}
C_1F-C_2\sum f^i\kappa_i^2-C_3\geq 0.
\ee
Substituting \eqref{2.2.6'} into \eqref{3.8} we obtian
\[C_1F-C_2F^2-C_3\geq 0.\]
Hence, $F$ is bounded from above.
\end{proof}

\bigskip
\section{$C^2$ estimates}
\label{sec4}

In this section we will show that the principal curvatures of $\mM_t$ remain bounded along the flow. Due to the complication of terms involving the third derivatives of $\mM_t$, we need to introduce new techniques to carefully analyze them. These are the most difficult estimates in this paper.
We prove
\begin{theorem}
\label{th4.1}
Under the normalized flow \eqref{0.5}, there exists a constant $C$ depending on $\mM_0, r, n, u$ and $F$,such that
\[|A|\leq C.\]
\end{theorem}
\begin{proof}
First, we note that $H^2-|A|^2=2S_2>0.$ Therefore, in order to show the principal curvatures are bounded, we only need to show $H$ is bounded.
Let us consider
\[Q=\log H-\log(u-a),\,\,\mbox{where $a=\frac{1}{2}\min u$.}\]
If $Q$ achieves its maximum at an interior point $X_0\in\mathcal{M}_{t_0}$, then at this point we have
\[\frac{H_i}{H}-\frac{u_i}{u-a}=0,\]
and
\[\mL Q=\frac{\mL H}{H}-\frac{\mL u}{u-a}\geq 0.\]
We will choose a local orthonormal frame in the neighborhood of $X_0$ such that at $X_0$ we have $h_{ij}=\kappa_i\delta_{ij}.$
By \eqref{2.1.7} we obtain the evolution equation for $H,$
\be\label{4.1}
\begin{aligned}
\frac{\partial}{\partial t}H&=\Phi_{kk}+\Phi h^k_ih^i_k-H\\
&=\nabla_k(\al r^{\al-1}r_kF+r^\al F_k)+\Phi|A|^2-H\\
&=\lt[\al r^{\al-1}r_{kk}F+\al(\al-1)r^{\al-2}r_k^2F+2\al r^{\al-1}r_kF_k+r^\al F_{kk}\rt]+\Phi|A|^2-H.\\
\end{aligned}
\ee

Since $h_{iikk}=h_{kkii}+h_{kk}h^2_{ii}-h_{ii}h^2_{kk},$ we get
\be\label{4.2}
\begin{aligned}
F_{kk}&=F^{ii}h_{iikk}+F^{pq, rs}h_{pqk}h_{rsk}\\
&=F^{ii}\lt(h_{kkii}+h_{kk}h^2_{ii}-h_{ii}h^2_{kk}\rt)+F^{pq, rs}h_{pqk}h_{rsk}\\
&=F^{ii}H_{ii}+Hf^i\kappa_i^2-|A|^2F+F^{pq,rs}h_{pqk}h_{rsk}.
\end{aligned}
\ee
Hence,
\be\label{4.3}
\begin{aligned}
\mL H&=\frac{\partial}{\partial t}H-r^\al F^{ii}H_{ii}\\
&=\al r^{\al-1}r_{kk}F+\al(\al-1)r^{\al-2}r_k^2F+2\al r^{\al-1}r_kF_k\\
&+r^\al H\sum f^i\kappa_i^2+r^\al F^{pq, rs}h_{pqk}h_{rsk}-H.\\
\end{aligned}
\ee

By a straightforward calculation we have,
\[r^2=S,\,\,2rr_i=2\lt<X, \tau_i\rt>,\,\,\text{and}\,\, 2r_i^2+2rr_{ii}=2-2h_{ii}u.\]
Thus
\be\label{add3}r_{ii}=\frac{1-h_{ii}u-r_i^2}{r}.\ee

Substituting \eqref{add3} into \eqref{4.3} and combining with \eqref{3.6} we obtain,
\be\label{4.4}
\begin{aligned}
\mL Q&=\frac{1}{H}\big[\al r^{\al-2}F(n-H u-|\nabla r|^2)+\al(\al-1)r^{\al-2}F|\nabla r|^2\\
&+2\al r^{\al-1}r_kF_k+r^\al H\sum f^i\kappa_i^2-H+r^\al F^{pq, rs}h_{pqk}h_{rsk}\big]\\
&-\frac{1}{u-a}\lt(-2\Phi+u+\al r^{\al-1}F\lt<X, \nabla r\rt>+r^\al u\sum f^i\kappa_i^2\rt)\geq 0.
\end{aligned}
\ee

By Lemma \ref{lem1.1}, Lemma \ref{lem2.2}, Lemma \ref{lem3.1}, and Lemma \ref{lem3.2}, equation \eqref{4.4} implies
\be\label{4.5}
\begin{aligned}
&\frac{1}{H}\lt(C_1+2\al r^{\al-1}r_kF_k+r^\al F^{pq, rs}h_{pqk}h_{rsk}\rt)\\
&+C_2-\frac{r^\al a}{u-a}\sum f^i\kappa_i^2\geq 0.\\
\end{aligned}
\ee

Now since
\[\beta F^2=S_2=\sum\limits_{p<q}\kappa_p\kappa_q,\,\,\beta=S_2(1, \cdots, 1)= {n\choose 2},\]
we have
\[2\beta F^{rs}F^{pq}+2\beta FF^{pq, rs}=S_2^{pq, rs}.\]
Therefore
\[F^{pq,rs}h_{pqk}h_{rsk}=\frac{S_2^{pq, rs}h_{pqk}h_{rsk}}{2\beta F}-\frac{F_k^2}{F}.\]
Furthermore, for any $\lambda>0$ we have
\[2\al r^{\al-1}r_kF_k\leq\frac{\lambda r^\al F_k^2}{HF}+\frac{\al^2 r^{\al-2}r^2_kHF}{\lambda}.\]
Therefore, equation \eqref{4.5} becomes
\be\label{4.6}
\begin{aligned}
&\frac{1}{H}\lt\{\frac{\al^2 r^{\al-2}r^2_kHF}{\lambda}+r^\al\lt(1-\frac{\lambda}{H}\rt)F^{pq, rs}h_{pqk}h_{rsk}\rt.\\
&\lt.+\frac{\lambda r^\al}{H}\cdot\frac{S_2^{pq, rs}h_{pqk}h_{rsk}}{2\beta F}\rt\}+C_2-\frac{ar^\al}{u-a}\sum f^i\kappa_i^2\geq 0,\\
\end{aligned}
\ee
where we used $\frac{H}{n}\geq F\geq\frac{1}{C}.$

From now on, we assume that $\kappa_1\geq\kappa_2\geq\cdots\geq\kappa_n$ at $X_0\in\mathcal{M}_{t_0}.$ In order to show $H$ is bounded at $X_0,$
we only need to show $\kappa_1$ is bounded. We will prove it in two steps. First, we will show $|\kappa_i|$ is bounded for $i\geq 2$.
Then, we will use this result to show that $\kappa_1$ is bounded.

\textbf{Step\,1.} In this step, we will show that when $\kappa_1>0$ large,
we have $|\kappa_i|$ is bounded by $\tilde{C}_1=\tilde{C}_1(r, u, F, \beta)$ for $i\geq 2$.

Note that,
\be\label{add1}
\begin{aligned}
S_2^{pq, rs}h_{pqk}h_{rsk}&=\sum\limits_{p\neq q}h_{ppk}h_{qqk}-\sum\limits_{p\neq q}h^2_{pqk}\\
&=H_k^2-\sum\limits_ph_{ppk}^2-\sum\limits_{p\neq q}h^2_{pqk}.\\
\end{aligned}
\ee

By the virtue of earlier estimates, \eqref{4.6} can be written as
\be\label{4.7}
\frac{C_3}{\lambda}+\frac{r^\al}{H}\lt(1-\frac{\lambda}{H}\rt)F^{pq, rs}h_{pqk}h_{rsk}
+\frac{\lambda r^\al}{2H^2\beta F}H^2_k+C_2-\frac{ar^\al}{u-a}\sum f^i\kappa_i^2\geq 0.
\ee

Moreover, since $Q_k=0$ at $X_0\in\mathcal{M}_{t_0},$ we have
\be\label{4.7'}
\frac{H_k}{H}=\frac{u_k}{u-a}=\frac{\kappa_k\lt<X, \tau_k\rt>}{u-a}.
\ee

Let $\lambda=\eta f^1,$ we may choose $\eta=\eta(r,u, F, \beta)>0$ small such that
\[\frac{\lambda r^\al}{2H^2\beta F}H_k^2=\sum\limits_k\frac{\eta r^\al f^1\kappa^2_k\lt<X, \tau_k\rt>^2}{2(u-a)^2\beta F}<\frac{ar^\al}{2(u-a)}\sum f^i\kappa_i^2,\]
where we used $f^1\leq f^2\leq \cdots\leq f^n.$
Notice that $f^1=\frac{H-\kappa_1}{2\beta F}<CH;$ so we can also assume $\eta>0$ so small that $\frac{\lambda}{H}<\frac{1}{2}.$
By the concavity of $F,$ \eqref{4.7} becomes
\be\label{4.8''}
\frac{C_3}{\eta f^1}+C_2-\frac{ar^\al}{2(u-a)}\sum\limits_if^i\kappa_i^2\geq 0.
\ee
Now, if $\eta f^1\geq 1$ at $X_0\in\mathcal{M}_{t_0},$ then we have
\begin{align*}
C_2+C_3\geq c_0\sum_if^i\kappa_i^2&=\frac{c_0}{2\beta F}\sum_i(H-\kappa_i)\kappa_i^2\\
&=\frac{c_0}{2\beta F}(S_2H-3S_3),
\end{align*}
where $S_3=\sum\limits_{1\leq i_1<i_2<i_3\leq n}\kappa_{i_1}\kappa_{i_2}\kappa_{i_3}$
and $c_0=c_0(r, u)$ is a positive constant.
It follows that when $H>0$ large,
\[\frac{3c_0}{2\beta F}S_3\geq \frac{c_0}{2\beta F}S_2H-C_2-C_3>0.\]
By Lemma 3 in \cite{GQ}, we conclude that $|\kappa_j|\leq \frac{7(n-1)S_2}{5\kappa_1}$, for $j\geq 2;$ so if $\eta f^1\geq 1$ then step 1 would be done.
Therefore, in the following, we will always assume $\eta f^1<1.$ In this case, \eqref{4.8''}
can be written as
\be\label{4.8}
\frac{C_3}{\lambda}-\frac{ar^\al}{2(u-a)}\sum f^i\kappa_i^2\geq 0.
\ee
This yields
\[\frac{ar^\al}{2(u-a)}f^1f^n\kappa_n^2\leq C_4.\]
Notice that
\be\label{4.0}
\begin{aligned}
f^1f^n&=\frac{(H-\kappa_1)(H-\kappa_n)}{4\beta^2F^2}\\
&>C_5(H^2-\kappa_1H)=C_5\lt(\sum\limits_{i\geq 2}\kappa_i^2+S_2+S_2(\kappa|\kappa_1)\rt)\geq C_6.
\end{aligned}
\ee
Here, the first inequality comes from the assumption that $\kappa_n<0,$ since if $\kappa_n\geq 0,$
we can get \eqref{4.8'} from Lemma \ref{lem3.2} directly. The second inequality in \eqref{4.0} is trivial if $S_2(\kappa|\kappa_1)\geq 0.$
When $S_2(\kappa|\kappa_1)<0,$ since
\[\lt(\sum_{i\geq 2}\kappa_i\rt)^2=\sum_{i\geq 2}\kappa_i^2+2S_2(\kappa|\kappa_1),\] we have
\[\sum_{i\geq 2}\kappa_i^2+S_2(\kappa|\kappa_1)>0,\]
and the second inequality still holds.

Therefore, we get
\[C_7\kappa_n^2\leq C_4,\]
which yields
\be\label{4.8'}
\kappa_n^2\leq\tilde{C}_0.
\ee
Hence, at the point where $Q$ achieves its interior maximum, $|\kappa_n|$ is bounded from above.
In the following, we want to show that $|\kappa_i|$ is bounded from above for all $i\geq 2$.

Now we assume at $X_0,$
\[\kappa_1\geq\kappa_2\geq \cdots\geq\kappa_k>0\geq\kappa_{k+1}\geq\cdots\geq\kappa_n\geq-\sqrt{\tilde{C_0}},\]
then it is easy to see that when $\kappa_1>2n\sqrt{\tilde{C}_0}$ we have
\be\label{5.12'}f^i=\frac{H-\kappa_i}{2\beta f}>C\kappa_1,\,\,\mbox{for $i\geq 2.$}\ee

Substituting \eqref{5.12'} into equation \eqref{4.8} we get
\be\label{add2}0\leq\frac{C_3}{\eta f^1}-\frac{ar^\al}{2(u-a)}\sum\limits_{i=2}^{n}f^i\kappa_i^2\leq
\frac{C_3}{\eta f^1}-\frac{ar^\al}{2(u-a)}C\kappa_1(|A|^2-\kappa_1^2).\ee
Note that
\[f^1\kappa_1=\frac{(H-\kappa_1)\kappa_1}{2\beta f}>CH(H-\kappa_1),\]
so by equation \eqref{4.0} we have $f^1\kappa_1>c_1$, where $c_1$ only depends on $\beta$ and $F.$
Therefore, \eqref{add2} implies
\[C_3>C_8(|A|^2-\kappa_1^2),\]
which gives $|\kappa_i|<\tilde{C}_1$ for $i\geq 2.$

\textbf{Step\, 2.} So far, we have proved that at the maximum point of $Q,$ if $\kappa_1>2n\sqrt{\tilde{C}_0}$ large, then
for $i\geq 2$ we have $|\kappa_i|<\tilde{C}_1$ for some constant $\tilde{C}_1=\tilde{C}_1(r, F, u, \beta).$
So, in this step, we will always assume $|\kappa_i|$ is bounded for $i\geq 2.$
Let us go back to equation \eqref{4.6}.

Without loss of generality, we may assume $\frac{\lambda}{H}<1/2$ and $\lambda<1.$
(Later we will see that in this step, we will choose $\lambda=\eta(f^1)^{2/3}.$ By step 1, we know
that $|\kappa_i|$ is bounded for $i\geq 2.$ Therefore, $f^1=\frac{\sum_{i\geq 2}\kappa_i}{2\beta F}$ is bounded from above, so we can always choose
$\eta>0$ small such that $\lambda<1.)$
Then, equation \eqref{4.6} implies
\be\label{4.10}
\frac{r^\al}{2H}F^{pq, rs}h_{pqk}h_{rsk}+\frac{r^\al\lambda}{2H^2F\beta}S^{pq, rs}_2h_{pqk}h_{rsk}+\frac{C_3}{\lambda}
-\frac{ar^\al\sum f^i\kappa_i^2}{u-a}\geq 0,
\ee

By a well known algebraic Lemma (see Lemma 7 in \cite{GRW15} for example), we have
\begin{align*}
&F^{pq, rs}h_{pqk}h_{rsk}\leq\sum\limits_{p\neq q}\frac{f^p-f^q}{\kappa_p-\kappa_q}h^2_{pqk}\\
&\leq 2\sum\limits_{p>1}\frac{f^p-f^1}{\kappa_p-\kappa_1}h^2_{pp1}=-\frac{1}{\beta F}\sum\limits_{p>1}h^2_{pp1}\\
&\leq-\frac{1}{\beta F(n-1)}(H_1-h_{111})^2.\\
\end{align*}
Moreover, \eqref{add1} yields
\begin{align*}
&S^{pq, rs}_2h_{pqk}h_{rsk}\leq \sum_kH^2_k-\sum_{p,k}h^2_{ppk}\\
&\leq H_1^2-h_{111}^2+\sum\limits_{k=2}^nH^2_k\leq C_9H^2+(H_1^2-h^2_{111}).\\
\end{align*}
Here the last inequality comes from \eqref{4.7'} and $|\kappa_i|<\tilde{C}_1,\,\,i\geq 2.$

Combining with equation \eqref{4.10} we get
\be\label{4.11}
\begin{aligned}
&-\frac{r^\al}{2H\beta F(n-1)}(H_1-h_{111})^2+\frac{r^\al\lambda}{2H^2F\beta}[C_9H^2+(H_1^2-h_{111}^2)]\\
&+\frac{C_3}{\lambda}-\frac{ar^\al\sum f^i\kappa_i^2}{u-a}\geq 0.\\
\end{aligned}
\ee

Now, let $\tilde{a}=\eta_1(f^1)^{1/3}$ and $\lambda=\eta (f^1)^{2/3}$. Here, we first choose $\eta_1>0$ such that $\tilde{a}<1,$ then we choose $\eta>0$ such that
$\frac{\eta}{\eta_1^2}\leq\frac{1}{n-1}$. We will divide this into two cases.

Case 1. If at $X_0\in\mathcal{M}_{t_0},$ $|H_1-h_{111}|\geq |\tilde{a}H_1|,$ since $\lambda\leq\frac{\tilde{a}^2}{n-1}$
we have, when $H>1$
\[-\frac{(H_1-h_{111})^2}{n-1}+\frac{\lambda}{H}H_1^2\leq -\frac{\tilde{a}^2H^2_1}{n-1}+\frac{\tilde{a}^2H_1^2}{(n-1)H}<0.\]
Therefore, equation \eqref{4.11} yields
\[\frac{C_4}{\lambda}\geq C_{10}f^1\kappa^2_1.\]
Multiplying by $\lambda$ on both sides we get
\[(f^{1})^{5/3}\kappa_1^2\leq C_{11}.\]
Since $f^1\kappa_1\geq c_1$ we conclude that $\kappa_1\leq\tilde{C}_2.$

Case 2. If at $X_0\in\mathcal{M}_{t_0},$ $|H_1-h_{111}|<|\tilde{a}H_1|,$ then we have
\[|h_{111}|<(1+\tilde{a})|H_1|.\]
Thus
\[|H_1^2-h_{111}^2|<3\tilde{a}H_1^2.\]
Substituting the above inequality into \eqref{4.11} we get
\[\frac{3r^\al\lambda\tilde{a}}{2H^2F\beta}H^2_1+\frac{C_4}{\lambda}-\frac{ar^\al\sum f^i\kappa_i^2}{u-a}\geq 0.\]
By \eqref{4.7'} we can choose $\eta=\eta(r, u, F, \beta)>0$ so small that
\[\frac{3r^\al\lambda\tilde{a}}{2H^2F\beta}H_1^2<\frac{ar^\al f^1\kappa_1^2}{2(u-a)}.\]
Hence, we get
\[C_{12}> (f^1)^{5/3}\kappa_1^2\]
which yields $\kappa_1\leq\tilde{C}_2.$
This completes the proof of Theorem \ref{4.1}.
\end{proof}

\bigskip
\section{Converging to a sphere}
\label{sec5}
Section 6 and 7 are small modifications of Section 4 and 5 of \cite{LSW18}, for completeness,
we will include them here.

It is sometimes more convenient to study the equation for the quantity
\[\rho(\xi, t)=\log r(\xi, t).\]
By a straightforward calculation we have
\[a_{ij}=e^{-\rho}\lt(1+|\bn\rho|^2\rt)^{-1/2}\tilde{a}_{ij},\]
where
\[\tilde{a}_{ij}=\gamma_{il}(\delta_{lm}+\bn_l\rho\bn_m\rho-\bn_{lm}\rho)\gamma_{mj},\]
and
\[\gamma_{ij}=\delta_{ij}-\frac{\bn_i\rho\bn_j\rho}{(1+|\bn\rho|^2)^{1/2}(1+(1+|\bn\rho|^2)^{1/2})}=e^\rho\lt(g^{-\frac{1}{2}}\rt)^{ij}.\]
Therefore, $\rho$ satisfies the following equation
\be\label{5.2}
\begin{aligned}
\rho_t=\frac{r_t}{r}&=-r^{\al-1}\sqrt{1+|\bn\rho|^2}\sigma_2^{1/2}(a_{ij})+1\\
&=-e^{\rho(\al-2)}\sigma_2^{1/2}(\tilde{a}_{ij})+1.\\
\end{aligned}
\ee

In the rest of this section, we shall finish the proof of Theorem \ref{th0.1}.
\begin{lemma}
\label{lem5.1}
For $\al\geq 2,$ there exists $C$ and $\gamma$ depending only on $n, \al,$ and $\mM_0,$
such that
\be\label{5.3}
\max\limits_{\s^n}\frac{|\bn r(\cdot, t)|}{r(\cdot, t)}\leq Ce^{-\gamma t,}\,\,\forall t>0.
\ee
\end{lemma}
\begin{proof}
Consider the auxiliary function
\[G=\frac{1}{2}|\bn\rho|^2,\,\,\text{where}\,\,\rho=\log r.\]
At the point where $G$ attains its spatial maximum, we have
\be\label{5.4}
0=\bn_iG=\sum\rho_l\rho_{li}
\ee
and
\be\label{5.5}
0\geq\bn_{ij}G=\rho_l\rho_{lij}+\sum\rho_{li}\rho_{lj}.
\ee
Moreover,
\be\label{5.6}
\begin{aligned}
G_t&=\sum\rho_l\rho_{lt}\\
&=\sum\rho_l\lt[-e^{\rho(\al-2)}(\al-2)\rho_l\sigma_2^{1/2}(\tilde{a}_{ij})-e^{\rho(\al-2)}F^{ij}\tilde{a}_{ijl}\rt]\\
&=-e^{\rho(\al-2)}\lt[(\al-2)\sigma_2^{1/2}(\tilde{a}_{ij})|\bn\rho|^2+F^{ij}\tilde{a}_{ijl}\rho_l\rt],\\
\end{aligned}
\ee
where $F^{ij}=\frac{\partial\sigma_2^{1/2}(\tilde{a}_{ij})}{\partial\tilde{a}_{ij}}.$
At the point under consideration, since $\sum_l\rho_l\rho_{li}=0,$ we have
\[\rho_r\bn_r\tilde{a}_{ij}=-\gamma_{il}\rho_r\bn_r\rho_{lm}\gamma_{mj}.\]
By the Ricci identity, we have
\[\bn_r\rho_{lm}=\bn_m\rho_{lr}+\delta_{lr}\rho_m-\delta_{lm}\rho_r.\]
Thus by \eqref{5.5},
\begin{align*}
\rho_r\bn_r\tilde{a}_{ij}&=-\gamma_{il}\rho_r(\bn_m\rho_{lr}+\delta_{lr}\rho_m-\delta_{lm}\rho_r)\gamma_{mj}\\
&\geq-\gamma_{il}(-\rho_{rl}\rho_{rm}+\rho_l\rho_m-\delta_{lm}|\bn\rho|^2)\gamma_{mj}.\\
\end{align*}
Substituting this into equation \eqref{5.6} we get
\be\label{5.7}
\begin{aligned}
G_t&\leq-e^{\rho(\al-2)}(\al-2)\sigma_2^{1/2}(\tilde{a}_{ij})|\bn\rho|^2\\
&+e^{\rho(\al-2)}F^{ij}\lt(-\gamma_{il}\rho_{rl}\rho_{rm}\gamma_{mj}+\gamma_{il}\rho_l\rho_m\gamma_{mj}-\sum_l\gamma_{il}\gamma_{lj}|\bn\rho|^2\rt)\\
&\leq e^{\rho(\al-2)}\lt(F^{ij}\gamma_{il}\rho_l\rho_m\gamma_{mj}-F^{ij}\sum_l\gamma_{il}\gamma_{lj}|\bn\rho|^2\rt).\\
\end{aligned}
\ee

Now let $A^{lm}=F^{ij}\gamma_{il}\gamma_{mj},$ by Theorem \ref{th4.1} we have
$\max\kappa[A^{lm}]-\sum A^{kk}\leq-C,$ thus
\[G_t\leq e^{\rho(\al-2)}\lt(A^{lm}\rho_l\rho_m-\sum^{kk}A^{kk}|\bn\rho|^2\rt)\leq-\gamma G,\]
for some positive constant $\gamma,$ this proves the Lemma.
\end{proof}

From \eqref{5.3} and Lemma \ref{lem1.1} we conclude that $|\bn r|\rightarrow 0$ exponentially as $t\rightarrow\infty.$ Hence by Theorem \ref{th4.1} and interpolation inequality we conclude that $r$ converges exponentially to a constant in the $C^\infty$ topology as $t\rightarrow\infty.$ This completes the proof of Theorem \ref{th0.1}.

\bigskip
\section{Counter examples}
\label{sec6}
In this section, we show that if $\al<2$ the flow
\be\label{6.1}
\lt\{
\begin{aligned}
\frac{\partial X}{\partial t}(x, t)&=-\sigma_2^{1/2}r^\al\nu,\\
X(\cdot, 0)&=X_0.
\end{aligned}
\rt.
\ee
may have unbounded ratio of radii, namely
\be\label{6.2}
R(X(\cdot, t))=\frac{\max_{\s^n}r(\cdot, t)}{\min_{\s^n}r(\cdot, t)}\goto\infty\,\,
\mbox{as $t\goto T$, for some $T>0$.}
\ee
Let $X(\cdot, t)$ be a convex solution to \eqref{6.1}, then its support function $u$
satisfies the equation
\be\label{6.3}
\lt\{
\begin{aligned}
\frac{\partial u}{\partial t}(x, t)&=-r^\al\left[\frac{S_{n-2}(\bn_{ij}u+u\delta_{ij})}{{n\choose 2}S_n(\bn_{ij}u+u\delta_{ij})}\right]^{1/2}\\
u(\cdot, 0)&=u_0.\\
\end{aligned}
\rt.
\ee
By a translation of time, we show below that there is a sub-solution $Y(\cdot, t)$
for $t\in(-1, 0)$ such that \eqref{6.2} holds as $t\goto 0.$ More precisely, we will construct a convex sub-solution
$Y(\cdot, t)$ such that its support function $\omega$ satisfies
\be
\lt\{
\begin{aligned}
\frac{\partial \omega}{\partial t}(x, t)&\geq-r^\al\left[\frac{S_{n-2}(\bn_{ij}\omega+\omega\delta_{ij})}{{n\choose 2}S_n(\bn_{ij}\omega+\omega\delta_{ij})}\right]^{1/2}\\
\omega(\cdot, 0)&= \omega_0.\\
\end{aligned}
\rt.
\ee
Moreover, we will show that $\min_{\mathbb{S}^n}\omega(\cdot, t)\goto 0$ while
$\max_{\mathbb{S}^n}\omega(\cdot, t)$ remains positive as $t\goto 0.$
\begin{lemma}
\label{lem6.1}
There is a sub-solution $Y(\cdot, t),$ where $t\in (-1, 0),$ to equation
\be\label{6.4}
\lt\{
\begin{aligned}
\frac{\partial u}{\partial t}(x, t)&=-ar^\al\left[\frac{S_{n-2}(\bn_{ij}u+u\delta_{ij})}{{n\choose 2}S_n(\bn_{ij}u+u\delta_{ij})}\right]^{1/2}\\
u(\cdot, 0)&=u_0\\
\end{aligned}
\rt.
\ee
for a sufficiently large constant $a>0,$ such that $\min_{\s^n}\omega(\cdot, t)\goto 0$
but $\max_{\s^n}\omega(\cdot, t)$ remains positive as $t\goto 0.$
\end{lemma}
\begin{proof}
Let $\hat{M}_t$ be the graph of the function
\be\label{6.5}
\phi(\rho, t)=\lt\{
\begin{aligned}
&-|t|^\theta+|t|^{-\theta+\sigma\theta}\rho^2,\,\,\mbox{if $\rho<|t|^\theta$}\\
&-|t|^\theta-\frac{1-\sigma}{1+\sigma}|t|^{\theta(1+\sigma)}+\frac{2}{1+\sigma}\rho^{1+\sigma},\,\,
\mbox{if $|t|^\theta\leq\rho\leq 1$},\\
\end{aligned}
\rt.
\ee
where $x\in\R^n,$ $\rho=|x|,$ $\sigma=\frac{q\theta-1}{\theta,}$ $q=2-\al,$
and $\theta>\frac{1}{q}$ is a constant.

When $0\leq\rho\leq|t|^\theta,$ by a straight forward calculation we have
\[\phi_i=2|t|^{\theta(\sigma-1)}x_i,\]
\[\phi_{ij}=2|t|^{\theta(\sigma-1)}\delta_{ij},\]
and
\[h_{ij}=\frac{2|t|^{\theta(\sigma-1)}\delta_{ij}}{\sqrt{1+4|t|^{2\theta(\sigma-1)}\rho^2}}.\]
The principal curvatures of the graph $\phi$ are
\[\kappa_1=\kappa_2=\cdots=\kappa_{n-1}=\frac{2|t|^{\theta(\sigma-1)}}{\sqrt{1+4|t|^{2\theta(\sigma-1)}\rho^2}}
\,\,\mbox{and $\kappa_n=\frac{2|t|^{\theta(\sigma-1)}}{\lt(\sqrt{1+4|t|^{2\theta(\sigma-1)}\rho^2}\rt)^3}$}.\]
Therefore,
\[r^\al\sigma_2^{1/2}\geq \frac{C|t|^{\al\theta}\cdot|t|^{\theta(\sigma-1)}}{1+4|t|^{2\theta(\sigma-1)}\rho^2}\geq C|t|^{\theta-1}\]
and $\lt|\frac{\partial}{\partial t}Y(p, t)\rt|\leq2\theta|t|^{\theta-1},$ where $p=(x, \phi(|x|, t))$
is a point on the graph of $\phi.$

When $|t|^\theta\leq\rho\leq 1,$ we obtain
\[\phi_i=2\rho^{\sigma-1}x_i,\]
\[\phi_{ij}=2(\sigma-1)\rho^{\sigma-3}x_ix_j+2\rho^{\sigma-1}\delta_{ij},\]
and
\[\kappa_1=\kappa_2=\cdots=\kappa_{n-1}=\frac{2\rho^{\sigma-1}}{\sqrt{1+4\rho^{2\sigma}}}\,\,
\mbox{and $\kappa_n=\frac{2\sigma\rho^{\sigma-1}}{\lt(\sqrt{1+4\rho^{2\sigma}}\rt)^3}$}.\]
Therefore, we have
\[r^\al\sigma_2^{1/2}>C\rho^{\al+\sigma-1}=C\rho^{1-\frac{1}{\theta}}\geq C|t|^{\theta-1}\]
and $\lt|\frac{\partial}{\partial t}Y(p, t)\rt|\leq2\theta|t|^{\theta-1}.$
Hence, the graph of $\phi$ is a sub-solution to \eqref{6.4}, provided $a$ is sufficiently large.

Next, we extend the graph of $\phi$ to a closed convex hypersurface $\hat{\mM}_t,$
such that it is $C^{1,1}$ smooth, uniformly convex, rotationally symmetric, and depends smoothly on $t.$
Moreover, we may assume that the ball $B_1(z)$ is contained in the interior of
$\hat{M}_t,$ for all $t\in(-1, 0),$ where $z=(0,\cdots, 0, 10)$ is a point on the
$x_{n+1}$- axis. Then $\hat{\mM}_t$ is a subsolution to \eqref{6.4} for sufficiently large $a$.
\end{proof}

We are in a position to prove Theorem \ref{th0.2}. For a given $\tau\in(-1, 0),$ let $\mM_{-1}$ be a smooth, closed, uniformly convex hypersurface
inside $\hat{\mM}_{\tau}$ and enclosing $B_1(z).$ Let $\mM_t$ be the solution to the flow \eqref{6.4} with initial data $\mM_{-1}.$
By Lemma \ref{lem6.1} and the classic comparison principle we have $\mM_t$ touches the origin at $t=t_0,$ for some $t_0\in(\tau, 0).$ We choose
$\tau$ to be very close to $0$ so that $|t_0|$ is sufficiently small.

On the other hand, let $\tilde{X}(\cdot, t)$ be the solution to
\[\frac{\partial\tilde{X}}{\partial t}=-ba\tilde{r}^\al\sigma_2^{1/2}\nu\]
with initial condition $\tilde{X}(\cdot, \tau)=\partial B_1(z),$ where $b=2^\al\sup\{|p|^\al: p\in \mM_t, \tau<t<t_0\}$
and $\tilde{r}=|\tilde{X}-z|$ is the distance from $z$ to $\tilde{X}.$ We can choose $\tau$ so close to $0$ such that
$B_{1/2}(z)$ is contained in $\tilde{X}(\cdot, t)$ for all $t\in(\tau, t_0).$ By the comparison principle, we see that the ball
$B_{1/2}(z)$ is contained in the interior of $\mM_t,$ for all $t\in(\tau, t_0).$ Therefore, as $t\rightarrow t_0$,
we have $\min r(\cdot, t)\rightarrow 0$ and $r(\cdot, t)>|z|=10.$ This proves \eqref{6.2} for $\mM_t.$

So far, we have proved Theorem \ref{th0.2} when $r^\al\sigma_2^{1/2}$ is replaced by $ar^\al\sigma_2^{1/2}$ for a large constant $a>0.$
Making the rescaling $\tilde{\mM}_t=a^{-\frac{1}{2-\al}}\mM_t,$ one easily verifies that $\tilde{\mM}_t$ solves the flow equation \eqref{6.1} and Theorem
\ref{th0.2} is proved.

\end{document}